\documentclass[11pt]{amsart}
\usepackage{amsmath}
\usepackage{amssymb}
\usepackage{amsthm}
\usepackage{bbm}
\usepackage{fullpage}
\usepackage{mathrsfs}
\usepackage{tikz-cd}
\usepackage{arydshln}
\usepackage{graphicx}
\usepackage{pdflscape}
\usepackage{multicol}
\usepackage{multirow}
\usepackage{mathtools}
\usepackage{hyperref}
\usepackage{enumitem}
\usepackage[toc,page]{appendix}
\usepackage[english]{babel}
\usepackage{setspace}
\usepackage{marvosym}
\allowdisplaybreaks[1]

\DeclareMathOperator{\Spn}{Span}

\DeclareMathOperator{\End}{End}

\DeclareMathOperator{\diag}{diag}

\newcommand{\ZeeII}{\mathbb{Z}_{2}}

\newcommand*{\Trn}{^{\mkern-2mu\mathsf{t}}}
\newcommand*{\sTrn}{^{\mkern-2mu\mathsf{st}}}

\newcommand{\gl}{\mathfrak{gl}}
\newcommand{\slx}{\mathfrak{sl}}
\newcommand{\Imn}{\mathbb{I}_{m|n}}

\newcommand{\gk}{\mathfrak{g}}

\newcommand{\hk}{\mathfrak{h}}

\newcommand{\Pk}{\mathfrak{P}}

\newcommand{\R}{\mathbb{R}}
\newcommand{\C}{\mathbb{C}}
\newcommand{\N}{\mathbb{N}}
\newcommand{\Z}{\mathbb{Z}}
\newcommand{\K}{\mathrm{K}}

\newcommand{\ev}{_{\overline{0}}}
\newcommand{\od}{_{\overline{1}}}

\newcommand{\Pt}{\tilde{P}}
\newcommand{\Qt}{\tilde{Q}}
\newcommand{\pt}{\tilde{p}}
\newcommand{\qt}{\tilde{q}}
\newcommand{\xt}{{\tilde{x}}}
\newcommand{\xit}{{\tilde{\xi}}}
\newcommand{\eps}{{\epsilon}}
\newcommand{\epst}{{\tilde{\epsilon}}}
\newcommand{\etat}{{\tilde{\eta}}}
\newcommand{\alphat}{{\tilde{\alpha}}}
\newcommand{\hkt}{\tilde{\hk}}
\newcommand{\thetat}{\tilde{\theta}}
\newcommand{\kappat}{\tilde{\kappa}}
\newcommand{\oGr}{\widetilde{Gr}}
\newcommand{\cW}{\mathcal{W}}
\newcommand{\cWR}{\mathcal{W}_{\mathbb{R}}}
\newcommand{\kP}{\mathcal{P}}
\newcommand{\A}{\mathbb{A}_{m|n}}

\newcommand{\commentout}[1]{}

\newcommand{\eqn}[1]{Eq.~(#1)}

\theoremstyle{remark}

\theoremstyle{plain}
\newtheorem{thm}{Theorem}[section]
\newtheorem*{thm*}{Theorem}
\newtheorem{main}[thm]{Theorem}
\newtheorem{prop}[thm]{Proposition}%[section]
\newtheorem*{prop*}{Proposition}
\newtheorem{lem}[thm]{Lemma}
\newtheorem*{lem*}{Lemma}

\theoremstyle{definition}

\newtheorem{rmk}[thm]{Remark}
\newtheorem{cor}[thm]{Corollary}

\usepackage{fancyhdr}
\title{Super Krawtchouk Polynomials via Lie Superalgebras}

\author{Plamen ILIEV}
\address{School of Mathematics, Georgia Institute of Technology\\
Atlanta, GA, 30332, USA}
\email{iliev@math.gatech.edu}
\thanks{P.~Iliev gratefully acknowledges support by a grant from the Simons Foundation.}

\author{Songhao ZHU}
\address{School of Mathematics, Georgia Institute of Technology\\
Atlanta, GA, 30332, USA}
\email{zhu.math@gatech.edu}

\date{\today}

\begin{document}

\begin{abstract}
    Multivariate extensions of the Krawtchouk polynomials have been studied by numerous authors in recent decades by exploring new connections to probability, representation theory and quantum integrability. We develop a theory of multivariate super Krawtchouk polynomials using the representation theory of the general linear Lie superalgebra, extending results of the first author in the classical setting. Specifically, in the present work we generalize the classical Krawtchouk polynomials, prove their orthogonality, construct certain recurrence relations, and discuss their connections with zonal spherical functions arising from a fermionic Fock-space framework in quantum mechanics.
\end{abstract}

\maketitle
%\tableofcontents

\section{Introduction}

Krawtchouk polynomials are a remarkable family of discrete orthogonal polynomials first introduced by Krawtchouk in 1929 and have attracted attention in various fields for almost a century since they were conceived. 
These polynomials are orthogonal with respect to the binomial distribution, and appear on the level of ${}_2F_1$ in the Askey scheme of hypergeometric polynomials \cite{KLS}.
In 1971, Griffiths \cite{G71} studied their multivariate analogs with respect to the multinomial distribution, largely in the context of probability theory. During the last 25 years these polynomials have appeared in different branches of pure and applied mathematics and were related to character algebras and Gelfand--Aomoto hypergeometric functions \cite{MT04}, representations of Lie algebras \cite{I12} and Lie groups \cite{GVZ13}, Markov chains and statistical modeling \cite{DG14}, quantum integrable systems \cite{IX20} and Gaudin models \cite{I24}.

In particular, the first author provided a Lie theoretic interpretation of these multivariate Krawtchouk polynomials in \cite{I12} proving that these polynomials can be constructed as the coefficients of transition matrices between two weight bases of certain $\slx$-modules. This approach was used to obtain a proof of the orthogonality relations and to derive new recurrence relations parametrized by Cartan subalgebras of the special linear Lie algebra.
A natural question is whether there is a family of similarly orthogonal polynomials in the super setting, and if so, does a comparable Lie theoretic interpretation exist? 

In this paper, we give positive answers to these questions.
Our main objects, super Krawtchouk polynomials $\kP$, are defined using a certain generating function in commuting variables and anticommuting variables in \eqn{\ref{eqn:mxPower}}. 
The parameters depend on tuples $\K = (p, \pt, U)$ and $\Lambda = (q, \qt, V)$ respectively for the commuting (even) part and the anticommuting (odd) part. %Both tuples are motivated by the original probability context definition in \cite{G71}.
The polynomials $\kP = \kP(\alpha, \alphat, \eps, \epst; \K, \Lambda, D )$ can be viewed as polynomials in $(m+n)$ independent variables $(\alphat_1, \dots ,\alphat_m; \epst_1, \dots ,\epst_n)$, parameterized by $(m+n)$ degree indices $(\alpha_1, \dots ,\alpha_m; \eps_1, \dots ,\eps_n)$, which can be placed on an equal footing. 
%In particular, for the anticommuting part, we give a determinant-type formula reflecting such symmetry.

To generalize \cite{I12}, we work with the general linear Lie superalgebra $\gk := \gl(m+1|n+1)$ and its two Cartan subalgebras $\hk$ and $\hkt$ introduced in Section~\ref{sec:def}.
We study the natural degree $D$ superpolynomial $\gk$-module $\Pk^D$ in commuting and anticommuting variables. 
With our initial data $\K$ and $\Lambda$, we define two monomial bases $\{x^\beta\xi^\eta\}$ and $\{\xt^\beta\xit^\eta\}$ of $\Pk^D$ which are shown to be $\hk$- and $\hkt$-weight bases. 

Our first main result, Theorem~\ref{thm:A} in Section~\ref{sec:thms}, says that the super Krawtchouk polynomials $\kP$ arise as the coefficients in the transition matrices between the two weight bases $\{x^\beta\xi^\eta\}$ and $\{\xt^\beta\xit^\eta\}$, thereby offering a superization of the first author's result.

Theorem~\ref{thm:B} establishes the orthogonality in our super setting. The proof requires some preparation. 
One key point is that our bilinear form $\langle -,- \rangle$, defined to be diagonal in $\{x^\beta\xi^\eta\}$ (\eqn{\ref{eqn:defProd}}), is in fact diagonal in $\{\xt^\beta\xit^\eta\}$ as well; this is the content of Proposition~\ref{prop:tForm}. 
To show this, we prove that $\langle -,- \rangle$ is contravariant with respect to the $\gk$-action on $\Pk^D$ in the sense of Proposition~\ref{prop:sign}. We note that the form is not supersymmetric in the usual sense; rather, we introduce a multi-color $\ZeeII^{n+1}$-grading on both $\gk$ and $\Pk^D$. Lemma~\ref{lem:colorSign} states that the multi-color grading is consistent with the parity $\ZeeII$-grading, and can be viewed as a refinement of it. 
A direct computation involving a certain antiautomorphism $\varphi$ of $\gk$ then establishes Proposition~\ref{prop:tForm}. This antiautomorphism is defined using the supertranspose.
In Section~\ref{sec:thms}, we discuss recurrence relations of the purely odd part of $\kP$, denoted by $\kP\od$. These are given in Proposition~\ref{prop:rec} and present an interpretation of these relations in $\bigwedge^{d}[\xi_j]$ using commuting differential operators therein.

The odd part of the theory has a different flavor from its classical counterpart. A main outcome of Section~\ref{sec:zonal} is a new interpretation of the odd Krawtchouk polynomials $\kP\od$ in terms of zonal spherical functions and fermionic quantum states. 
Specifically, \eqn{\ref{eqn:krZonal}} shows that $\kP\od$ is given by $q$- and $\qt$-weighted evaluations of the zonal spherical function $\phi_D$ on the oriented Grassmannian $\oGr(n+1,D)$, arising naturally from the fermionic Fock space picture. 
Combined with a probabilistic interpretation in \eqn{\ref{eqn:prob}}, which arises from a quantum determinantal point process \cite{Gottlieb07}, it shows that the squared norms of these evaluations encode the transition probabilities between the occupation bases determined by $\Lambda=(q,\qt,V)$.

The paper is structured as follows. 
%We outline our paper as follows.
In Section~\ref{sec:def}, we give the definitions of the super Krawtchouk polynomials and introduce the general linear superalgebra. 
In Section~\ref{sec:rep} we define the superpolynomial $\gk$-modules $\Pk^D$ and the bilinear form, and show its contravariance. 
%In Section~\ref{sec:thms} we derive the orthogonality with the help of the previous results and discuss recurrence relation for the purely odd Krawtchouk polynomials. 
In Section~\ref{sec:thms} we derive our main results and discuss recurrence relations for the odd Krawtchouk polynomials.
Finally in Section~\ref{sec:zonal} we relate the super Krawtchouk polynomials to the spherical functions in the setting of quantum mechanics. 
%Section~\ref{sec:thms} and Section~\ref{sec:zonal} are independent. 

\section{Definitions} \label{sec:def}
In this section, we define the Krawtchouk polynomials and present the Lie superalgebra framework. 
Throughout the paper we set $\N = \Z_{\geq 0}$ and denote the set $\{0, 1, \dots, k\}$ by $[k]$ for $k\in \N$. We also write $\ZeeII:= \{0,1\}$.
For $(k+1)$-tuples $\alpha = (\alpha_0, \alpha_1, \dots, \alpha_k)$, $\beta = (\beta_0, \beta_1, \dots, \beta_k)$, and an index set $I = \{i_0, i_1, \dots, i_d\}$ we define
\[
|\alpha| := \sum_{i=0}^k \alpha_i, \quad \alpha!:= \prod_{i=0}^k\alpha_i! , \quad \alpha^\beta := \prod_{i=0}^k \alpha_i^{\beta_i}  , \quad  \alpha_I := \prod_{j=0}^d\alpha_{i_j}.
\]

\subsection{Krawtchouk polynomials} \label{subsec:KP}
We begin with the classical even case, and introduce the odd and mixed cases afterward. 
Griffiths’ construction of the Krawtchouk polynomials is based on a generating function. %For our purposes, we give the following adaptation of the generating function definition from \cite{I12}. 
We give the following definition adapted from \cite{I12}.

We first introduce some data. Let $\K := (p, \pt, U)$ be a tuple consisting of two scalar tuples $p = (p_0, p_1, \dots, p_m), \pt = (\pt_0,\pt_1, \dots, \pt_m)\in \C^{m+1}$ and a complex $(m+1)\times (m+1)$ matrix $U$ such that 
\begin{enumerate}[label=(\roman*)]
    \item $p_0 = \pt_0 \neq 0$;
    \item $U = (u_{i,j})_{i,j=0}^m$ satisfies the normalization condition $u_{0,j}=u_{i,0}=1$ for $i, j \in [m]$;
    \item we have the following matrix equation
    \begin{equation} \label{eqn:PU}
        PU\Pt U\Trn = p_0 I_{m+1}
    \end{equation}
    where $P = \diag(p)$, $\Pt = \diag(\pt)$.
\end{enumerate}
As an immediate consequence, we have $|p| = |\pt| =1$.
%The probabilities mentioned in the introduction are thus special cases of such tuples. \underCons

Let $\{x_i: i\in [m]\}$ be $m+1$ commuting indeterminates. Let $D \in \N$ be a fixed degree. Let $\alpha = (\alpha_0, \alpha_1,\dots,\alpha_m), \alphat= (\alphat_0, \alphat_1,\dots,\alphat_m) \in \N^{m+1}$ be such that $|\alpha| = |\alphat| = D$. The even Krawtchouk polynomials $\kP\ev(\alpha, \alphat) = \kP\ev(\alpha, \alphat; \K, D)$ are defined as follows:
\begin{equation} \label{eqn:evenPower}
    \prod_{i=0}^m\left(\sum_{j=0}^m u_{i,j} x_j\right)^{\alphat_i}
=
\sum_{\alpha}
\frac{D!}{\alpha!}\kP\ev(\alpha,\alphat )
x^{\alpha}.
 \end{equation}
%Substituting $\alphat_0 = D -\sum_{i\geq1}\alphat_i$ and $\alpha_0 = D -\sum_{i\geq 1}\alpha_i$, we may regard $\{\alpha_1 , \dots, \alpha_m, \alphat_1, \dots, \alphat_m \}$ as $2m$ independent variables. \underCons 
Substituting $\alphat_0 = D -\sum_{i\geq 1}\alphat_i$, we may regard $\{\alphat_1, \dots, \alphat_m \}$ as $m$ independent variables. 
%A priori $\kP\ev$ are functions of $\alphat$ parameterized by the degree parameters $\alpha$, but they turn out to be polynomials. 
In \cite{MT04}, an explicit formula for these polynomials is given, and the roles of $\alpha$ and $\alphat$ are completely symmetric; see \cite[Eq.~(1.3)]{I12}. 
To obtain the formulation in \cite{I12}, simply put $x_0 = 1$ together with a straightforward relabeling. 

The odd analog is given similarly. Let $\Lambda := (q, \qt, V)$ be another tuple as above where $q, \qt\in \C^{n+1}$ and $V$ is an $(n+1)\times (n+1)$ matrix. Define $Q:=  \diag(q_j)$, $\Qt := \diag(\tilde{q}_j)$ for $j \in [n]$. We require $q_0=\qt_0\neq 0$, $V = (v_{i,j})$ with $v_{0,j}=v_{i,0}=1$ and
\begin{equation} \label{eqn:QV}
    QV\Qt V\Trn = q_0 I_{n+1}.
\end{equation}
Let $\{\xi_j:j\in [n]\}$ be the Grassmann variables subject to the following commuting rules:
\begin{align*}
    \xi_i \xi_j = - \xi_j \xi_i, \quad &  \text{ for all } i, j\in [n]; \\
    \xi_i x_k = x_k \xi_i, \quad  & \text{ for all } i \in [n], k\in [m].
\end{align*}
We omit the wedge $\wedge$ between two variables whenever the anticommutativity is clear.
Consequently $\xi_j^2 = 0$ for all $j\in [n]$.

Let $d\in \N$ be such that $d\leq n+1$. Let $\eps, \epst\in \ZeeII^{n+1}$ satisfy $|\eps|=|\epst| = d$. The odd Krawtchouk polynomials $\kP\od(\eps, \epst) = \kP\od(\eps, \epst; \Lambda, d)$ are defined as follows:
\begin{equation} \label{eqn:oddPower}
    \prod_{i=0}^n\left(\sum_{j=0}^n v_{i,j} \xi_j\right)^{\tilde \eps_i}
=
\sum_{\eps}
d!\kP\od(\eps,\tilde \eps)
\xi^{\eps},
\end{equation}
where the products are taken in the normal order $\xi_0^{\eps_0}\cdots \xi_n^{\eps_n}$. Here we note that $\eps! = 1$ for all $\eps\in \ZeeII^{n+1}$.
%This formulation parallels the even case.
A priori $\kP\od$ are just functions depending on $\eps, \epst$, and parameters $q, \qt, V$ and $d$. 
%These $\kP\od$, however, are again polynomials. 
%To see this,
To see that $\kP\od$ are polynomials, observe that on the left side, only the factors with $\epst_i=1$ contribute non-trivially. Denote the index set $\{i: \epst_i = 1\}$ by $I'$. The expansion gives a homogeneous degree $d$ polynomial in anticommuting products of $\xi_j$. Let $J = \{j_k:j_0 < \cdots < j_{d-1}\}\subseteq [n]$. Then the coefficient of $\xi_J = \xi_{j_0} \cdots \xi_{j_{d-1}}$ is precisely the determinant of the minor
%The coefficient of, say, $z_{j_1} z_{j_2}\cdots z_{j_D}$ is precisely the determinant of the minor
\[
V_{I',J} := \left(v_{i,j} \right)_{i\in I', j\in J}.
\]
Let $A(\eps, \epst) := \diag(\tilde\eps) V \diag(\eps)$. The only potentially non-zero $d\times d$ minor is $\det V_{I',J}$. 
%Thus for a uniform polynomial description of these coefficients, we have 
Thus
\begin{equation} \label{eqn:oddP}
    \kP\od (\eps, \epst) = \frac{1}{d!}\sum_{I, J} \det A_{I,J}(\eps, \epst),
\end{equation}
where the sum is taken over all size $d$ index subsets $I, J\subseteq [n]$. 
When $\eps, \epst$ are viewed as formal variables, $\kP\od$ becomes a homogeneous polynomial of degree $2d$ linear in each variable.

To connect to a Lie superalgebra point of view in the next section, we give a mixed version of the set-up where the data $\K$ and $\Lambda$ are concatenated. Let
\[
\Imn := \{0, 1, \dots, m, m+1, \dots, m+n+1\}.
\]
For two square matrices $A$ and $B$ we write
\[
(A|B) := \diag (A|B) = \begin{pmatrix}
    A & 0 \\ 0 & B 
\end{pmatrix}.
\]
Set $Y := (y_{i,j})_{i,j\in \Imn} = (U|V)$. Using the properties $P, \Pt, Q, \Qt$, we see that
\begin{equation*} %\label{eqn:mxData}
    (p_0^{-1}P|q_0^{-1}Q) Y (\Pt| \Qt)  Y\Trn = I_{m+n+2}.
\end{equation*}

We further define $\{w_k: k\in \Imn\}$ with $w_i := x_i$, for $ i\in [m]$ and $w_{m+1+j} := \xi_j$ for $j\in [n]$.
For $\alpha \in \N^{m+1}$ and $\eps \in \ZeeII^{n+1}$ we write $(\alpha, \eps)$ for the concatenated $(m+n+2)$-tuple. 
Let
\[
\A := \N^{m+1} \times \ZeeII^{n+1}, \quad \A^D := \{(\alpha, \eps)\in \A: |(\alpha, \eps)|=D\}.
\]
For $(\alphat, \epst)\in \A^D$ we write
\begin{equation} \label{eqn:mxPower}
    \prod_{i\in \Imn} \left(\sum_{j\in \Imn} y_{i,j} w_j\right)^{(\alphat, \epst)_i}
=
\sum_{(\alpha, \eps)\in \A^D }
\frac{D!}{\alpha!}\kP(\alpha, \eps, \alphat,\epst )
w^{(\alpha, \eps)}.
\end{equation}
Here $\kP(\alpha,\eps,\alphat,\epst ) = \kP(\alpha,\eps,\alphat,\epst; \K, \Lambda, D )$ is a polynomial by the above discussion. 
Note that the left side of \eqn{\ref{eqn:mxPower}} is the product of the left sides of \eqn{\ref{eqn:evenPower}} and \eqn{\ref{eqn:oddPower}}. Let $d = |\eps| = |\epst|$. Thus we must have
\begin{equation} \label{eqn:mxkP}
    \kP(\alpha, \eps,\alphat,\epst ) = \binom{D}{d}^{-1} \kP\ev (\alpha, \alphat;\K, D-d) \kP\od (\eps, \epst; \Lambda, d).
\end{equation}
Note that
$\kP(\alpha,\eps,\alphat,\epst)=0$ unless $|\alpha|=|\alphat|$ and $|\eps|=|\epst|$. For this reason, we further set
\[
\A^{D-d,d} := \{(\alpha, \eps)\in \A^D: |\eps|=d\}.
\]
Thus $\kP(\alpha,\eps,\alphat,\epst)=0$ unless $(\alpha,\eps), (\alphat,\epst)\in \A^{D-d,d}$.

\begin{rmk} \label{rmk:duality}
    We note that the space of parameters carries a natural duality involution.
    By \eqn{\ref{eqn:PU}} and \eqn{\ref{eqn:QV}}, the following map 
    \[
    (\K, \Lambda) = (p, \pt, U, q, \qt, V) \mapsto (\K', \Lambda') := (\pt, p, U\Trn, \qt, q, V\Trn)
    \]
    is an involution on the space of parameters.
    By \cite[Eq.~(1.3)]{I12} (cf. \cite{MT04}), 
    \[
    \kP\ev(\alpha, \alphat; \K, D) = \kP\ev(\alphat, \alpha; \K', D).
    \]
    Similarly in our setting, from \eqn{\ref{eqn:oddP}}, we see that transposing $V$ amounts to interchanging $\eps$ and $\epst$ in \eqn{\ref{eqn:oddPower}}, i.e.,
    \[
    \kP\od (\eps, \epst; \Lambda, d) = \kP\od (\epst, \eps; \Lambda', d).
    \]
    Using \eqn{\ref{eqn:mxkP}} we have
    \[
    \kP(\alpha,\eps,\alphat,\epst; \K, \Lambda, D ) = \kP(\alphat,\epst,\alpha,\eps; \K', \Lambda', D ).
    \]
\end{rmk}

\subsection{Lie superalgebras} \label{subsec:LSA}
Consider the vector superspace $\C^{m+1|n+1}$ with standard homogeneous basis indexed by $\Imn =\{0,1,\dots,m+n+1\}$.
We assign parity to $i\in \Imn$ by
\[
|i|=
\begin{cases}
0, & 0\le i\le m,\\
1, & m+1\le i\le m+n+1.
\end{cases}
\]
Relative to this basis, each element of \(\End(\C^{m+1|n+1})\) is represented by a block matrix
\[
\left( \begin{array}{c|c} A & B\\ \hline C & D\end{array}\right ).
\]
Let $E_{i,j}$ denote the matrix with 1 in the $(i,j)$-th entry and 0 elsewhere, and set
\begin{align*}
    A_{i,j} &:=E_{i,j}, \text{ for } i,j\in [m];  &&B_{i,j}:=E_{i,m+1+j}, \text{ for } i \in [m], j\in [n]; \\
    C_{i,j} &:=E_{m+1+i,j}, \text{ for } i \in [n], j\in [m];  && D_{i,j}:=E_{m+1+i,m+1+j}, \text{ for } i,j\in [n]. 
\end{align*}
The even part of \(\End(\C^{m+1|n+1})\) consists of the block diagonal matrices, while the odd part consists of the block off-diagonal matrices. Equipped with the supercommutator
\[
[X,Y]:=XY-(-1)^{|X||Y|}YX
\]
for homogeneous $X,Y\in \End(\C^{m+1|n+1})$, extended bilinearly, $\End(\C^{m+1|n+1})$ becomes the Lie superalgebra
\[
\gk:=\gl(m+1|n+1).
\]

Let $\hk$ denote the standard Cartan subalgebra of $\gk$ spanned by the matrices $E_{i, i}$ for all $i\in \Imn$.
Choose $\thetat$, $\kappat$ so that 
\[
R:= \thetat \Pt U\Trn, \quad S:= \kappat\Qt V\Trn
\]
satisfy $\det R = \det S = 1$. Let $\theta := (p_0 \thetat)^{-1}$ and $\kappa := (q_0 \kappat)^{-1}$. 
We further define for $X\in \gk$
\[
\tilde{X} := (R|S) X (R^{-1}|S^{-1}).
\]
The elements $\tilde{E}_{i,i}$ span a new Cartan subalgebra, which we denote by $\hkt$. 

On $\gk$ the supertranspose map is given by
\[
X = \left( \begin{array}{c|c}A & B\\ \hline C & D\end{array}\right )\mapsto X\sTrn := \left( \begin{array}{c|c}A\Trn & C\Trn \\ \hline -B\Trn & D\Trn\end{array}\right ).
\]
Define on $\gk$
\[
\varphi: X \mapsto  (\theta^{-1}\Pt| \kappa^{-1}\Qt) X\sTrn  (\theta^{-1}\Pt| \kappa^{-1}\Qt)^{-1}.
\]
Then $\varphi$ is a super antiautomorphism. %satisfying $\varphi^4 = \mathrm{Id}$.
The following lemma follows from direct calculations.
\begin{lem} \label{lem:phiCalc} 
We have the following equations:
    \begin{equation*}
    \varphi(E_{i,i}) = E_{i,i}, \quad \varphi(\tilde{E}_{i,i})=\tilde{E}_{i,i}, \quad \text{ for all }i\in \Imn,
\end{equation*}
and for all indices $i$ and $j$ for which the expression is defined,
\begin{alignat*}{4}
\varphi(A_{i,j})          &= \frac{\pt_j}{\pt_i}A_{j,i}, 
&\qquad \varphi(\tilde{A}_{i,j}) &= \frac{p_j}{p_i}\tilde{A}_{j,i};
&\qquad \varphi(B_{i,j})          &= -\frac{\theta\qt_j}{\kappa\pt_i}C_{j,i},
&\qquad \varphi(\tilde{B}_{i,j}) &= -\frac{\thetat q_j}{\kappat p_i}\tilde{C}_{j,i};\\
\varphi(C_{i,j})          &= \frac{\kappa\pt_j}{\theta\qt_i}B_{j,i},
&\qquad \varphi(\tilde{C}_{i,j}) &= \frac{\kappat p_j}{\thetat q_i}\tilde{B}_{j,i};
&\qquad \varphi(D_{i,j})          &= \frac{\qt_j}{\qt_i}D_{j,i},
&\qquad \varphi(\tilde{D}_{i,j}) &= \frac{q_j}{q_i}\tilde{D}_{j,i}.
\end{alignat*}
\end{lem}

We also record the following change of basis calculation for the two Cartan subalgebras $\hk$ and $\hkt$.

\begin{lem} \label{lem:cartanSwap}
%Let the notation be as above. 
For $i\in [m]$, we have
    \begin{align} 
        \tilde{A}_{i,i} &= p_0^{-1}p_i\sum_{0\leq k,l\leq m}\pt_k u_{i,k}u_{i,l} {A}_{k,l}, \label{eqn:CartanEv}\\
        A_{i,i} &= p_0^{-1}\pt_i\sum_{0\leq k,l\leq m}p_k u_{k,i}u_{l,i} \tilde{A}_{k,l}, \label{eqn:tCartanEv}
        %&= q_0^{-1}\qt_i \left(\sum_{0\leq k\neq l\leq n}q_k v_{k,i}v_{l,i} \tilde{E}_{k,l} + \sum_{0\leq k\leq n}q_k v_{k,i}v_{k,i} \tilde{E}_{k,k} \right)
    \end{align}
    and similarly for $i\in [n]$, 
    \begin{equation}\label{eqn:tCartan}
        \tilde{D}_{i,i} = q_0^{-1}q_i\sum_{0\leq k,l\leq n}\qt_k v_{i,k}v_{i,l} {D}_{k,l}, \quad D_{i,i} = q_0^{-1}\qt_i\sum_{0\leq k,l\leq n}q_k v_{k,i}v_{l,i} \tilde{D}_{k,l}. %\notag \\
        %&= q_0^{-1}\qt_i \left(\sum_{0\leq k\neq l\leq n}q_k v_{k,i}v_{l,i} \tilde{E}_{k,l} + \sum_{0\leq k\leq n}q_k v_{k,i}v_{k,i} \tilde{E}_{k,k} \right)
    \end{equation}
\end{lem}
\begin{proof}
    \eqn{\ref{eqn:CartanEv}} follows from a direct matrix computation by the definition of the mapping $X\mapsto \tilde{X}$.
    For \eqn{\ref{eqn:tCartanEv}}, one may use the duality in Remark~\ref{rmk:duality}, or directly as follows. Write $A_{i,i} = \sum a_{k,l} \tilde{A}_{k,l}$ for scalars $a_{k,l}$, then we observe that we must have $R^{-1} A_{i,i} R = \sum a_{k,l} A_{k,l}$ with which we may calculate $a_{k,l}$. The proof of \eqn{\ref{eqn:tCartan}} is similar.
\end{proof}

\section{Polynomial representations and a bilinear form} \label{sec:rep}
In this section we introduce the polynomial modules for $\gk = \gl(m+1|n+1)$ and define a bilinear form used to show the orthogonality theorem. 
Let us also point out that in \cite{I12}, $\mathfrak{sl}$ was used instead of $\gl$. The difference is inessential. 

\subsection{Polynomial representations}
As in the previous section, let $\{x_i:i\in [m]\}$ be $m+1$ commuting variables and $\{\xi_j:j\in [n]\}$ be $n+1$ anticommuting variables. % such that $x_i\xi_j = \xi_j x_i$ for all suitable indices $i$ and $j$. 
We also set 
\[
w_i := x_i, \text{ for } i\in [m]; \quad w_{m+1+j} := \xi_j \text{ for } j\in [n]; \quad \partial_i = \partial_{w_i} \text{ for } i\in \Imn.
\]
Let $\Pk$ denote the supersymmetric algebra on the variables $\{w_i : i\in\Imn \}$. Explicitly,
\[
\Pk = \C[x_i:i\in [m]] \otimes \bigwedge[\xi_j: j\in [n]].
\]
For $(\alpha,\eps)\in \A^D $, we let $\deg x^\alpha \xi^\eps := |\alpha| + |\eps|$. We further define
\[
\Pk^D :=\Spn \{x^\alpha \xi^\eps: (\alpha, \eps)\in \A^D\}.
\]

The Lie superalgebra $\gk$ acts on $\Pk$ by the following standard differential operators:
\begin{equation} \label{eqn:Eijaction}
    \rho: E_{i,j} \mapsto w_i \partial_j.
\end{equation}
For convenience, we write $X.p$ for $\rho(X)(p)$ for all $X\in \gk$ and $p\in \Pk$.
The operator $\partial_j$ satisfies the graded Leibniz rule
\begin{equation} \label{eqn:Leibniz}
    \partial_j (ab) = \partial_j (a)b+(-1)^{|j||a|}a\partial_j(b),
\end{equation}
for homogeneous $a,b\in \Pk$.
Let $u_i$ (resp. $v_i$) be the standard basis vector with 1 in the $i$-th entry and 0 elsewhere for $\N^{m+1}$ (resp. $\ZeeII^{n+1}$), and for $\eps\in \ZeeII^{n+1}$ set
\[
s_j(\eps) := \sum_{k=0}^{j-1} \eps_k.
\] 
Consequently, we have the following explicit action of $\gk$ on monomials $x^\alpha \xi^\eps$ by \eqn{\ref{eqn:Eijaction}}.
% \begin{align*}
%     x_i\partial_{x_j} x^\alpha \xi^\eps &= \alpha_j x^{\alpha+u_i-u_j} \xi^{\eps},\\
%     x_i\partial_{\xi_j} x^\alpha \xi^\eps &= (-1)^{s_j(\eps)} \eps_j x^{\alpha+u_i} \xi^{\eps-v_j}, \\
%     \xi_i \partial_{x_j} x^\alpha \xi^\eps &= (-1)^{s_i(\eps)} \alpha_j x^{\alpha-u_j} \xi^{\eps+v_i}, \\
%     \xi_i \partial_{\xi_j} x^\alpha \xi^\eps &= (-1)^{s_j(\eps)+s_i(\eps-v_j)} \eps_j x^{\alpha} \xi^{\eps+v_i-v_j}.
% \end{align*}
\begin{equation} \label{eqn:EijAct}
    T\,x^\alpha \xi^\eps=
    \begin{cases}
    \alpha_j x^{\alpha+u_i-u_j}\xi^\eps,
    & T=x_i\partial_{x_j},\quad i,j\in [m],\\[4pt]
    (-1)^{s_j(\eps)}\eps_j\,x^{\alpha+u_i}\xi^{\eps-v_j},
    & T=x_i\partial_{\xi_j},\quad i\in [m],\ j\in [n],\\[4pt]
    (-1)^{s_i(\eps)}\alpha_j\,x^{\alpha-u_j}\xi^{\eps+v_i},
    & T=\xi_i\partial_{x_j},\quad i\in [n],\ j\in [m],\\[4pt]
    (-1)^{s_j(\eps)+s_i(\eps-v_j)}\eps_j\,x^\alpha\xi^{\eps+v_i-v_j},
    & T=\xi_i\partial_{\xi_j},\quad i,j\in [n].
    \end{cases}
\end{equation}
Here we recall that for $\eps_i \geq 2$, the monomial $\xi^\eps$ vanishes by the anticommutativity. 

For later use, we give the following identity. 
\begin{lem} \label{lem:leibSign}
    For all $\eps\in \ZeeII^{n+1}$ and $i,j\in [n]$, we have
\begin{equation}
    s_i(\eps-v_j)+s_j(\eps) = s_i(\eps+v_i-v_j)+s_j(\eps-v_j).
\end{equation}
\end{lem}
\begin{proof}
    Note that $s_j(\eps)$ only records entries with indices less than $j$. Thus, $s_j(\eps) = s_j(\eps-v_j)$ and $s_i(\eps-v_j) = s_i(\eps+v_i-v_j)$. The result follows.
\end{proof}

We give the following standard homogeneous module decomposition for $\Pk$.

\begin{prop}
    We have 
    \[
    \Pk = \bigoplus_{D\in \N} \Pk^D
    \]
    as a $\gk$-module decomposition where each component $\Pk^D$ is irreducible.
\end{prop}
\begin{proof}
    The action defined in \eqn{\ref{eqn:Eijaction}} preserves the total degree, showing that each $\Pk^D$ is a submodule. 
    %For each $\Pk^D$, each monomial is a weight vector of the Cartan subalgebra $\hk$. Moreover, the weight spaces are all one dimensional. Thus if $M\subseteq \Pk^D$ is a submodule, hence a weight module, it must contain a monomial.   
    To see that each $\Pk^D$ is irreducible, observe that each monomial $x^\alpha\xi^\eps\in \Pk^D$ is an $\hk$-weight vector, and the corresponding weight space is one dimensional. Thus any non-zero submodule \(M\subseteq \Pk^D\) contains a monomial $x^\alpha\xi^\eps$. Then by applying suitable $E_{i,j}$'s we may recover all monomials of the same degree. 
\end{proof}

Using $R$ and $S$ from Subsection~\ref{subsec:LSA}, we introduce a change of variables:
\begin{equation}
    (\xt_0, \dots, \xt_m | \xit_0, \dots, \xit_n) := (x_0, \dots, x_m | \xi_0, \dots, \xi_n)(R|S).
\end{equation}
It follows that
\begin{equation} \label{eqn:cov}
    \xt_i = \thetat\sum_{k = 0}^m \pt_k u_{ik}x_k, \quad \xit_j = \kappat\sum_{k = 0}^n \qt_k v_{jk}\xi_k.
\end{equation}
%Moreover, the following proposition says the monomials in $\xt$ and $\xit$ form a weight space basis for the Cartan subalgebra $\hkt$. 

\begin{prop} \label{prop:tEijAct}
    %Let the notation be as above. Then
    The basis $\{\xt^\alpha \xit^\eps\}$ is an $\hkt$-weight basis for $\Pk^D$. Explicitly,
    \begin{equation} \label{eqn:tEijAct}
    X.\xt^\alpha \xit^\eps=
    \begin{cases}
    \alpha_j \xt^{\alpha+u_i-u_j}\xit^\eps,
    & X=\tilde{A}_{i,j} ,\quad i,j\in [m],\\[4pt]
    (-1)^{s_j(\eps)}\eps_j\,\xt^{\alpha+u_i}\xit^{\eps-v_j},
    & X = \tilde{B}_{i,j},\quad i\in [m],\ j\in [n],\\[4pt]
    (-1)^{s_i(\eps)}\alpha_j\,\xt^{\alpha-u_j}\xit^{\eps+v_i},
    & X = \tilde{C}_{i,j},\quad i\in [n],\ j\in [m],\\[4pt]
    (-1)^{s_j(\eps)+s_i(\eps-v_j)}\eps_j\,\xt^\alpha\xit^{\eps+v_i-v_j},
    & X = \tilde{D}_{i,j},\quad i,j\in [n].
    \end{cases}
\end{equation}
\end{prop}
\begin{proof}
    In \eqn{\ref{eqn:cov}}, let $M = (R|S)$. Then we have $\tilde{w} = (\xt_0, \dots, \xt_m|\xit_0, \dots, \xit_n) = wM$, and
    \[
    \tilde{w}_i = \sum_{k\in \Imn} w_k M_{k,i}, \quad \partial_{\tilde{w}_j} = \sum_{l\in \Imn} (M^{-1})_{j,l}\partial_{w_{l}},
    \]
    from which it follows that
    \[
    \tilde{w}_i \partial_{\tilde{w}_j} = \sum_{k,l\in \Imn}M_{k,i}(M^{-1})_{j,l} w_k\partial_{w_l}.
    \]
    On the other hand, $\tilde{E}_{i,j} = ME_{i,j}M^{-1}$, so
    \[
    \tilde{E}_{i,j} = \sum_{k,l\in \Imn} M_{k,i}(M^{-1})_{j,l} E_{k,l}.
    \]
    Thus by \eqn{\ref{eqn:Eijaction}} we have $\rho(\tilde{E}_{i,j})= \tilde{w}_i \partial_{\tilde{w}_j}$ and \eqn{\ref{eqn:tEijAct}} now follows from direct computations.
\end{proof}

\subsection{A bilinear form}
Recall $\A := \N^{m+1} \times \ZeeII^{n+1}$ and $ \A^D := \{(\alpha, \eps)\in \A: |(\alpha, \eps)|=D\}$ from above.
We introduce the following bilinear form on $\Pk^D$:
\begin{equation} \label{eqn:defProd}
    \langle x^\alpha\xi^\eps, x^\beta \xi^\eta \rangle := \delta_{\alpha, \beta}\delta_{\eps, \eta} \frac{\alpha ! }{\tilde{p}^\alpha \tilde{q}^\eps} \theta^{|\alpha|}\kappa^{|\eps|}, \text{ for all } (\alpha, \eps), (\beta, \eta)\in \A^D.
\end{equation}

\begin{rmk}
    The form in \eqn{\ref{eqn:defProd}} is a natural extension of the form defined in \cite[\eqn{4.1}]{I12}. However, this differs from a typical supersymmetric bilinear form. In particular, for a supersymmetric bilinear form, every homogeneous vector of odd parity has zero self-pairing. Instead, the above form allows all basis monomials to have non-zero norm. 
\end{rmk}

We aim to show the following proposition, parallel to \cite[Lemma~4.1]{I12}. The proof requires some preparation concerning the action of $\gk$ on $\Pk^D$, and is thus postponed until later in this section. 

\begin{prop} \label{prop:tForm}
    %Let notation be as above. We have
    We have 
    \[
    \langle \xt^\alpha\xit^\eps, \xt^\beta \xit^\eta \rangle = \delta_{\alpha, \beta}\delta_{\eps, \eta} \frac{\alpha ! }{{p}^\alpha {q}^\eps} {\thetat}^{|\alpha|}{\kappat}^{|\eps|},
    \]
    for all $(\alpha, \eps), (\beta, \eta)\in \A^D$. 
\end{prop}

The Lie superalgebra $\gk = \gl(m+1|n+1)$ is equipped with the standard $\ZeeII$-grading. There is, however, a finer description. Let $v_k$ denote the $k$-th standard basis vector of $\ZeeII^{n+1}$. Whenever $k<0$, the symbol $v_k$ is regarded as 0.
Also recall from Subsection~\ref{subsec:LSA} that we defined parity for $\Imn$. 
For $E_{i,j} \in \gk$, we set
\[
\overline{E_{i,j}} := |i| v_{i-(m+1)} +|j|v_{j-(m+1)}.
\]
Equivalently, we may start by setting $\deg x_i = 0$ and $\deg \xi_j = v_j \in \ZeeII^{n+1}$ on $\C^{\Imn} = \Spn\{x_i, \xi_j\}$. As $\gk$ is identified with $(\C^{\Imn})^*\otimes \C^{\Imn}$, the odd part $\gk\od$ inherits the grading described above.
Explicitly,
\[
\overline{E_{i,j}}= \begin{cases}
    0, & |i|= |j| = 0\\
    v_{i-(m+1)}+v_{j-(m+1)}, & |i| = |j| = 1\\
    v_{i-(m+1)},  & |i| = 1, |j| = 0\\
    v_{j-(m+1)} & |i| = 0, |j| = 1\\
\end{cases} \;.
\]

We define $\gk_{\eps} := \Spn \{E_{i,j}: \overline{E_{i,j}} = \eps\}$ for $\eps\in \ZeeII^{n+1}$.
Thus
\[
\gk = \bigoplus_{\eps\in \ZeeII^{n+1}} \gk_\eps.
\]
This grading records the odd indices as a vector in $\ZeeII^{n+1}$ modulo 2. From a physics point of view, it records which fermionic positions occur in $E_{i,j}$ modulo 2.
This multi-color $\ZeeII^{n+1}$-grading is consistent with the usual $\ZeeII$-super grading in the following sense.

\begin{lem} \label{lem:colorSign}
    For all $\eps, \eta\in \ZeeII^{n+1}$, we have $[\gk_\eps , \gk_\eta ] \subseteq \gk_{\eps+\eta}$.
\end{lem}
\begin{proof}
    It suffices to consider the elementary matrices $E_{i,j}$. Let $\overline{E_{i,j}} = \eps$ and $\overline{E_{k,l}} = \eta$. The supercommutator says
    \begin{equation} \label{eqn:supCom}
        [E_{i,j}, E_{k,l}] = \delta_{j,k}E_{i,l}-(-1)^{(|i|+|j|)(|k|+|l|)} \delta_{i,l}E_{k,j}.
    \end{equation}
    When $j=k$ the first term is non-zero, and we see that 
    \[
    \eps + \eta = |i| v_{i-(m+1)} +2|j|v_{j-(m+1)} + |l| v_{l-(m+1)} = |i| v_{i-(m+1)} + |l| v_{l-(m+1)} = \overline{E_{i,l}}.
    \]
    Similarly, for the second term when $i=l$, we get $\eps + \eta =  \overline{E_{k,j}}$.
    Thus the right side of \eqn{\ref{eqn:supCom}} always has multi-color degree $\eps+\eta$, proving the claim. 
\end{proof}

\begin{rmk}
    The above multi-color grading is a refinement of the $\ZeeII$-parity. Indeed, the latter is recovered by the map $\ZeeII^{n+1} \rightarrow \ZeeII: \eta \mapsto |\eta| \pmod{2}$. In the sense of \cite{Scheunert79}, this gives a $\ZeeII^{n+1}$-graded $\pi$-Lie algebra structure where $\pi: \ZeeII^{n+1} \rightarrow \C^\times$ is the character given by $ \eta \mapsto (-1)^{|\eta|}$; compare \cite{Price97}.
\end{rmk}

We may also consider the multi-color grading on our module $\Pk^D$.
Let $x^\alpha\xi^\eps \in \Pk^D$. We set
\[
\overline{x^\alpha\xi^\eps} := (\eps_0, \dots, \eps_n) \in \ZeeII^{n+1}.
\]
For the action of $E_{i,j}$ on these monomials, it is not hard to see that
\[
\overline{E_{i,j}.x^\alpha\xi^\eps} = \overline{E_{i,j}} + \overline{x^\alpha\xi^\eps}.
\]
By Lemma~\ref{lem:colorSign}, this also extends to all of $\gk$.
Now we are ready to introduce a proper generalization of \cite[Eq.~(4.2)]{I12}.
For $\eps, \eta \in \ZeeII^{n+1}$, we write $\eps \cdot \eta = \sum_{i=0}^n \eps_{i}\eta_{i}$. 
Recall $\varphi$ from Subsection~\ref{subsec:LSA}.

\begin{prop}\label{prop:sign}
For all homogeneous $X\in \gk$ and $u\in \Pk^D$, and $v\in \Pk^D$, we have 
    \begin{equation} 
    \langle X . u,\, v \rangle
  = (-1)^{|X|\overline{X}\cdot \overline{u}}\langle u, \varphi(X). v \rangle.
\end{equation}
\end{prop}

\begin{proof}
    It suffices to check the action of $X = E_{i,j}$ on monomials $x^\alpha\xi^\eps$ when the form is non-vanishing. When $i=j$, the equation follows from $|X|=0$ and that $\hk$ is invariant under $\varphi$.
    When $i\neq j$, we proceed case by case with our ``$A, B, C, D$'' notation introduced in Subsection~\ref{subsec:LSA}. 

    When $X = A_{i,j}$, the left side is
    \begin{align*}
        \langle A_{i,j}. x^\alpha \xi^\eps, x^{\alpha+u_i-u_j}\xi^{\eps} \rangle &=  \alpha_j\langle x^{\alpha+u_i-u_j}\xi^{\eps } , x^{\alpha+u_i-u_j}\xi^{\eps} \rangle \\
        &=  \frac{(\alpha+u_i)!}{\pt^\alpha \qt^\eps \pt_i \pt^{-1}_j} \theta^{|\alpha|}\kappa^{|\eps|},
    \end{align*}
    while the right side is
    \begin{align*}
        \langle  x^\alpha \xi^\eps, \varphi(A_{i,j}).x^{\alpha+u_i-u_j}\xi^{\eps} \rangle &= \frac{\pt_j}{\pt_i}\langle  x^{\alpha}\xi^{\eps } , A_{j,i}.x^{\alpha+u_i-u_j}\xi^{\eps} \rangle \\
        &=  \frac{\pt_j}{\pt_i} (\alpha_{i}+1)\langle  x^{\alpha}\xi^{\eps } , x^{\alpha}\xi^{\eps} \rangle \\
        &=\frac{\pt_j}{\pt_i} \frac{(\alpha+u_i)!}{\pt^\alpha \qt^\eps } \theta^{|\alpha|}\kappa^{|\eps|},
    \end{align*}
    thus equal.
    
    When $X = B_{i,j}$, which acts by $x_i\partial_{\xi_j}$, we assume $\eps_j = 1$; otherwise the identity is vacuously true. The left side is
    \begin{align*}
        \langle B_{i,j}. x^\alpha \xi^\eps, x^{\alpha+u_i}\xi^{\eps-v_j} \rangle = (-1)^{s_j(\eps)} \frac{(\alpha+u_i)!}{\pt^\alpha \qt^\eps \pt_i \qt^{-1}_j} \theta^{|\alpha|+1}\kappa^{|\eps|-1},
    \end{align*}
    while for $\eps_j = 1$ the sign on the right side is $(-1)^{|B_{i,j}|\overline{B_{i,j}}\cdot \overline{x^\alpha \xi^\eps}} = (-1)^{1( 1\cdot 1)} = -1$. Hence the right side becomes
    \begin{align*}
    -\langle x^\alpha \xi^\eps, \varphi(B_{i,j}).x^{\alpha+u_i}\xi^{\eps-v_j} \rangle &=  \frac{\theta\qt_j}{\kappa\pt_i}  \langle  x^{\alpha}\xi^{\eps} , C_{j,i}.x^{\alpha+u_i}\xi^{\eps-v_j} \rangle \\
    %&= \frac{\theta\qt_j}{\kappa\pt_i} (\alpha_i+1)(-1)^{s_j(\eps)} \langle  x^{\alpha}\xi^{\eps} , x^{\alpha}\xi^{\eps} \rangle \\
    & =  \frac{\theta\qt_j}{\kappa\pt_i} (-1)^{s_j(\eps)}\frac{(\alpha+u_i)!}{\pt^\alpha \qt^\eps } \theta^{|\alpha|}\kappa^{|\eps|},
    %\\
    %& =  \langle B_{i,j}. x^\alpha \xi^\eps, x^{\alpha+u_i}\xi^{\eps-v_j} \rangle.
    \end{align*}
    equal to the left side.
    
    When $X = C_{i,j}$, which acts by $\xi_i \partial_{x_j}$, we similarly assume $\eps_i = 0$. The left side is
    \begin{equation*}
        \langle C_{i,j}. x^\alpha \xi^\eps, x^{\alpha-u_j}\xi^{\eps+v_i} \rangle = (-1)^{s_i(\eps)} \alpha_j \frac{(\alpha-u_j)!}{\pt^\alpha \qt^\eps \pt^{-1}_j \qt_i} \theta^{|\alpha|-1}\kappa^{|\eps|+1},
    \end{equation*}
    while for $\eps_i = 0$ the sign on the right side is $(-1)^{|C_{i,j}|\overline{C_{i,j}}\cdot \overline{x^\alpha \xi^\eps}} = (-1)^{1( 1 \cdot 0)} = 1$, 
    so as in the previous case, we have
    \[
    \langle x^\alpha \xi^\eps, \varphi(C_{i,j}).x^{\alpha - u_j}\xi^{\eps +v_i} \rangle = \frac{\kappa\pt_j}{\theta\qt_i} (-1)^{s_i(\eps)}\frac{\alpha!}{\pt^\alpha \qt^\eps } \theta^{|\alpha|}\kappa^{|\eps|} = \langle C_{i,j}. x^\alpha \xi^\eps, x^{\alpha-u_j}\xi^{\eps+v_i} \rangle.
    \]
    % \begin{align*}
    % \langle x^\alpha \xi^\eps, \varphi(C_{i,j}).x^{\alpha - u_j}\xi^{\eps +v_i} \rangle &=  \frac{\kappa\pt_j}{\theta\qt_i}  \langle  x^{\alpha}\xi^{\eps} , B_{j,i}.x^{\alpha - u_j}\xi^{\eps +v_i} \rangle \\
    % &= \frac{\kappa\pt_j}{\theta\qt_i} (-1)^{s_i(\eps)} \langle  x^{\alpha}\xi^{\eps} , x^{\alpha}\xi^{\eps} \rangle \\
    % & =  \frac{\kappa\pt_j}{\theta\qt_i} (-1)^{s_i(\eps)}\frac{\alpha!}{\pt^\alpha \qt^\eps } \theta^{|\alpha|}\kappa^{|\eps|} \\
    % & =  \langle C_{i,j}. x^\alpha \xi^\eps, x^{\alpha-u_j}\xi^{\eps+v_i} \rangle.
    % \end{align*}

    Finally when $X = D_{i,j}$, which acts by $\xi_i\partial_{\xi_j}$, we assume $\eps_i=0$ and $\eps_j = 1$. 
    The left side is 
    \begin{equation*}
        \langle D_{i,j}. x^\alpha \xi^\eps, x^{\alpha}\xi^{\eps+v_i-v_j} \rangle = (-1)^{s_j(\eps)+s_i(\eps-v_j)}  \frac{\alpha!}{\pt^\alpha \qt^\eps \qt^{-1}_j \qt_i} \theta^{|\alpha|}\kappa^{|\eps|}.
    \end{equation*}
    The sign on the right side is now $1$ since $|D_{i,j}|=0$. So the right side is
    \begin{align*}
    \langle x^\alpha \xi^\eps, \varphi(D_{i,j}).x^{\alpha}\xi^{\eps+v_i-v_j} \rangle &=  \frac{\qt_j}{\qt_i}  \langle  x^{\alpha}\xi^{\eps} , D_{j,i}.x^{\alpha}\xi^{\eps+v_i-v_j} \rangle \\
    %&= \frac{\qt_j}{\qt_i} (-1)^{s_i(\eps+v_i-v_j)+s_j(\eps-v_j)} \langle  x^{\alpha}\xi^{\eps} , x^{\alpha}\xi^{\eps} \rangle \\
    & =  \frac{\qt_j}{\qt_i}  (-1)^{s_i(\eps+v_i-v_j)+s_j(\eps-v_j)} \frac{\alpha!}{\pt^\alpha \qt^\eps } \theta^{|\alpha|}\kappa^{|\eps|},
    %\\
    %& =  \langle D_{i,j}. x^\alpha \xi^\eps, x^{\alpha}\xi^{\eps+v_i-v_j} \rangle,
    \end{align*}
    and is equal to the left side by Lemma~\ref{lem:leibSign}.

    The proposition is now established. 
\end{proof}

Now we are ready to show Proposition~\ref{prop:tForm}.

\begin{proof}[Proof of Proposition~\ref{prop:tForm}]
    Let $|\eps| = |\eta| = d$. Note that the form in \eqn{\ref{eqn:defProd}} is the product of the even part involving $x_i$'s and the form defined on the odd $\xi_j$'s:
    \begin{equation}\label{eqn:oddForm}
        \langle \xi^\eps, \xi^\eta \rangle = \delta_{\eps, \eta} \frac{\kappa^{d}}{\tilde{q}^\eps}.
    \end{equation}
    For the even part, see \cite[Lemma~4.1]{I12}. It suffices to show that \eqn{\ref{eqn:oddForm}} implies
    \begin{equation}\label{eqn:oddTForm}
        \langle \xit^\eps, \xit^\eta \rangle = \delta_{\eps, \eta} \frac{\tilde{\kappa}^{d}}{q^\eps}.
    \end{equation}
    Suppose $\eps \neq \eta$. By Proposition~\ref{prop:tEijAct}, $\xit^\eps$ and $\xit^\eta$ lie in distinct weight spaces of $\hkt$. Since the form is contravariant in the sense of Proposition~\ref{prop:sign} and since every $H\in\hkt$ has parity $0$ and satisfies $\varphi(H)=H$, it follows that the left side of \eqn{\ref{eqn:oddTForm}} must vanish.

    Suppose instead $\eps = \eta$. 
    %For an index set $I = \{i_k: i_0<\cdots< i_{a}\}$, and a $(a+1)$-tuple $\alpha$ we write $\alpha_I$ for $\alpha_{i_0} \cdots \alpha_{i_a}$. 
    Using $\xit_i = \sum \xi_k S_{ki}$ and calculating directly, we see that $\langle \xit_i, \xit_j \rangle = \delta_{ij}{\kappat}/{q_i}$.
    The form in \eqn{\ref{eqn:oddForm}} can also be obtained by first setting $\langle \xi_i,\xi_j \rangle =\delta_{i,j} \kappa/ \qt_i$ then extending to the exterior algebra $\bigwedge^d [\xi_j]$ by $\langle u_0 \cdots u_{d-1},v_0 \cdots v_{d-1} \rangle := \det (\langle u_i ,v_j\rangle)$ for $u_i, v_j\in \Spn \{\xi_j\}$.
    Thus, for $\xit_{[d-1]} = \xit_0\cdots \xit_{d-1}$, we have
    \begin{equation} \label{eqn:base}
        \langle \xit_{[d-1]}, \xit_{[d-1]}  \rangle = \det (\langle \xit_i ,\xit_j\rangle) = \frac{\kappat^d}{q_{[d-1]}}.
    \end{equation}
    Now using $\varphi(\tilde{D}_{i,j})$ from Lemma~\ref{lem:phiCalc} and $|\tilde{D}_{i,j}|=0$ in Proposition~\ref{prop:sign},  we have
    \[
    \langle \tilde{D}_{i,j}.\xit^\eps, \xit^{\eps+v_i-v_j} \rangle = \frac{q_j}{q_i}\langle \xit^\eps, \tilde{D}_{j,i}.\xit^{\eps+v_i-v_j} \rangle,
    \]
    where we assume $\eps_i = 0$ and $\eps_j=1$. A direct computation using the definition and Lemma~\ref{lem:leibSign} shows
    \begin{equation} \label{eqn:induc}
        \langle \xit^{\eps+v_i-v_j}, \xit^{\eps+v_i-v_j} \rangle = \frac{q_j}{q_i}\langle \xit^\eps, \xit^{\eps} \rangle.
    \end{equation}
    Inductively, \eqn{\ref{eqn:base}} and \eqn{\ref{eqn:induc}} allow us to compute the squared norm of all other homogeneous $\xit^\eps$ with $|\eps|=d$ starting from $\xit_{[d-1]}$, proving \eqn{\ref{eqn:oddTForm}}.
\end{proof}

We also offer an alternative proof using the Cauchy--Binet formula.

\begin{proof}[Alternative proof of Proposition~\ref{prop:tForm}]
    As in the first proof, it suffices to show \eqn{\ref{eqn:oddTForm}} for $\eps = \eta$. Set $J = \{j\in [n]: \eps_j = 1\}$. We see that 
    \[
    \xit^\eps = \kappat^d \prod_{j\in J}\sum_{k=0}^n \qt_k v_{jk}\xi_k = \kappat^d \sum_{|I|=d} \det(V\Qt)_{J,I} \xi^{\eps(I)},
    \]
    where the sum runs over all possible index subsets and $\eps(I)_j = 1$ if and only if $j\in I$. Thus by the orthogonality of the monomial basis with respect to \eqn{\ref{eqn:defProd}}, we have
    \begin{equation} \label{eqn:mid}
        \langle \xit^\eps, \xit^\eps \rangle = \kappat^{2d} \sum_{|I|=d} \left(\det(V\Qt)_{J,I} \right)^2 \kappa^d\qt^{-\eps(I)}.
    \end{equation}
    Recall that $\Qt = \diag (\qt_j)$. Thus, 
    \begin{align*}
        \sum_{|I|=d} \left(\det(V\Qt)_{J,I} \right)^2 {\qt^{-\eps(I)}} &=  \sum_{|I|=d} \det(V\Qt)_{J,I} \det(V)_{J,I} \\
        &= \sum_{|I|=d} \det(V\Qt)_{J,I} \det(V\Trn)_{I,J} \\
        &= \det(V\Qt V\Trn)_{J, J},
    \end{align*}
    where the last equality follows from the Cauchy--Binet formula (see, e.g., \cite{ThMat}). Since $V\Qt V\Trn = q_0Q^{-1}$, this value is $q_0^d q^{-\eps}$. Substituting into \eqn{\ref{eqn:mid}} and using $\kappa \kappat = q_0 ^{-1}$ proves \eqn{\ref{eqn:oddTForm}}.
\end{proof}

\section{Orthogonality and difference equations} \label{sec:thms}
In this section, we state and prove the main results using the theory developed above. In particular, we show that the super Krawtchouk polynomials occur as entries in transition matrices between bases for the Lie superalgebra modules $\Pk^D$, give an extended form of the classical orthogonality result, and finally discuss an odd parallel of the bispectral difference equations in \cite{I12}. 

\subsection{Orthogonality}
Recall that on $\Pk^D$ we defined two monomial bases $\{x^\alpha \xi^\eps\}$ and $\{\xt^\alpha \xit^\eps\}$, which are respectively $\hk$- and $\hkt$-weight bases, and showed that with respect to the bilinear form \eqn{\ref{eqn:defProd}} they are both orthogonal. 
Parallel to \cite[Theorem~5.1]{I12}, we have the following proposition. Recall $\kP\ev$, $\kP\od$, $\kP$ and the tuples $\K = (p, \pt, U)$ and $\Lambda = (q, \qt, V)$ from Subsection~\ref{subsec:KP}.

\begin{prop} \label{prop:transOdd}
Let $\eps, \epst\in \ZeeII^{n+1}$ with $|\eps| = |\epst|=d$. Write $\kP\od(\eps, \epst) = \kP\od(\eps, \epst; \Lambda, d)$. We have
    \[
    \xit^{\epst} = \kappat^d d! \sum_{|\eps|=d} \kP\od(\eps, \epst) \qt^\eps \xi^\eps, \quad \xi^{\eps} = \kappa^d d! \sum_{|\epst|=d} \kP\od(\eps, \epst) q^{\epst} \xit^{\epst}.
    \]
\end{prop}

\begin{proof}
    By \eqn{\ref{eqn:cov}} we have %and in parallel with the generating-function definition above, we may write
    \[
    \xit^{\epst} = \prod_{i=0}^n \left(\kappat\sum_{j=0}^n \qt_j v_{ij}\xi_j\right)^{\epst_i} = \sum_{\eps} M(\eps, \epst) \xi_0^{\eps_0}\xi_1^{\eps_1}\cdots \xi_n^{\eps_n}.
    \]
    We know from above (cf. \eqn{\ref{eqn:oddPower}}) that these coefficients are $M(\eps, \epst)=\sum_{I,J} \det B_{I,J}$ for $d\times d$ minors $B_{I,J}$ in $B = \diag(\epst) \kappat (\qt_j v_{ij}) \diag(\eps)$. Comparing with \eqn{\ref{eqn:oddP}} we must have
    \[
    M(\eps, \epst) = d! \kappat^d \qt^\eps\kP\od(\eps, \epst)
    \]
    proving the desired relation. The second identity is similar.
\end{proof}

Our first main result is as follows. 

\begin{main} \label{thm:A}
    The super Krawtchouk polynomials $\kP (\alpha, \eps, \alphat, \epst) = \kP(\alpha,\eps,\alphat,\epst; \K, \Lambda, D )$ arise, up to explicit scalars, as entries of the transition matrices between the $\hk$- and $\hkt$-weight bases for the module $\Pk^D$. Specifically, for $x^\alpha \xi^\eps$ and $\xt^\alphat \xit^\epst$ in $\Pk^D$, %letting $d := |\eps|=|\epst|$, we have
    with $(\alpha, \eps),(\alphat, \epst) \in \A^{D-d,d}$, we have 
    \begin{align}
        \xt^\alphat \xit^\epst &= \thetat^{D-d} \kappat^d D! \sum_{(\alpha, \eps)\in \A^{D-d,d}} \kP (\alpha, \eps, \alphat, \epst) \frac{\pt^\alpha \qt^\eps}{\alpha!} x^\alpha  \xi^\eps,  \label{eqn:A1}\\
        x^\alpha \xi^\eps &= \theta^{D-d} \kappa^d D! \sum_{(\alphat, \epst)\in \A^{D-d,d}} \kP (\alpha, \eps, \alphat, \epst) \frac{p^\alphat q^\epst}{\alphat!} \xt^\alphat  \xit^\epst. \label{eqn:A2}
    \end{align}
\end{main}
\begin{proof}
    By \eqn{\ref{eqn:EijAct}} and Proposition~\ref{prop:tEijAct}, the bases $\{x^\alpha\xi^\eps\}$ and $\{\xt^\alphat\xit^\epst\}$ are $\hk$- and $\hkt$-weight bases. By \cite[Theorem~5.1]{I12} we have 
    \[
    \xt^{\alphat} = \thetat^{D-d} (D-d)! \sum_{|\alpha|=D-d} \kP\ev(\alpha, \alphat) \frac{\pt^\alpha}{\alpha!} x^\alpha.
    \]
    The first identity follows from multiplying this identity by the first equation in Proposition~\ref{prop:transOdd} and the definition of $\kP$ in \eqn{\ref{eqn:mxkP}}. The second identity is similar. 
\end{proof}

\begin{rmk}
    The transition matrix between the tilde basis and the non-tilde basis is block diagonal. Specifically, ignoring the scalars in Theorem~\ref{thm:A}, this matrix is given by
    \[
    \left(
    \begin{array}{c:c:c}
    \kP\ev &  &  \\ \hdashline 
     & \kP = \kP\ev \kP\od &  \\ \hdashline
     &  &  \kP\od
    \end{array}
    \right).
    \]
    % \[
    % \diag(\kP\ev|\kP|\kP\od)
    % \]
    The upper left block corresponds to the theory in \cite{I12}, depicting the change of the $x/\xt$-bases. The lower right block gives the transition between the $\xi/\xit$-bases and represents the ``purely odd'' Krawtchouk polynomials we defined here. If $D > n+1$, then lower right block vanishes as there are no pure $\xi$ or $\xit$ monomials. The middle ``mixed'' block is also block diagonal; each block size is determined by the dimension of the subspace $\{x^\alpha \xi^\eps: |\eps|=d\}$, explicitly given by
    \[
    \binom{D-d+m}{m} \binom{n+1}{d}.
    \]
\end{rmk}

We give the following corollary of Theorem~\ref{thm:A} which expresses $\kP$ explicitly as the pairing between basis monomials under the bilinear form defined in \eqn{\ref{eqn:defProd}}.

\begin{cor} \label{cor:pairing}
%Let notation be as above. We have
For $(\alpha, \eps),(\alphat, \epst) \in \A^{D-d,d}$, we have
\[
\kP(\alpha, \eps, \alphat, \epst) = \frac{p_0^{D-d}q_0^d}{D!} \langle \xt^\alphat \xit^\epst, x^\alpha \xi^\eps \rangle.
\]
\end{cor}

In the classical setting, Krawtchouk polynomials are orthogonal with respect to the multinomial weights; see, e.g., \cite[Corollary~5.3]{I12}.
The next result gives a natural extension.
%of the classical orthogonality of Krawtchouk polynomials. 

\begin{main} \label{thm:B}
    Let $(\alpha, \eps),(\alphat, \epst), (\beta, \eta), (\tilde{\beta}, \etat)\in \A^{D-d,d}$. The super Krawtchouk polynomials satisfy the following orthogonality relations
    \begin{align*}
        \sum_{(\alpha, \eps)\in \A^{D-d,d}} \pt^\alpha\qt^\eps \frac{D!}{\alpha!} \kP(\alpha,\eps, \alphat, \epst)\kP(\alpha, \eps, \tilde{\beta}, \etat)  = \delta_{\alphat, \tilde{\beta}} \delta_{\epst, \etat} \frac{p_0^{D-d} q_0^d}{D!} \frac{\alphat!}{p^\alphat q^\epst}, \\
        \sum_{(\alphat, \epst)\in \A^{D-d,d}} p^\alphat q^\epst \frac{D!}{\alphat!} \kP(\alpha,\eps, \alphat, \epst)\kP( {\beta}, {\eta}, \alphat, \epst)  = \delta_{\alpha, {\beta}} \delta_{\eps, {\eta}} \frac{p_0^{D-d}q_0^d}{D!} \frac{\alpha!}{\pt^\alpha \qt^\eps}.
    \end{align*}
\end{main}

\begin{proof}
    We pair both sides in \eqn{\ref{eqn:A1}} with the monomial $\xt^{\tilde{\beta}}\xit^{\etat}$. The first identity then follows from Proposition~\ref{prop:tForm} and Corollary~\ref{cor:pairing}. Similarly the second identity follows from a dual argument, where \eqn{\ref{eqn:defProd}} is used instead of Proposition~\ref{prop:tForm}.
\end{proof}

\subsection{Difference equations}

We discuss the recurrence relations for the odd Krawtchouk polynomials in this subsection and offer an interpretation in the exterior algebra. Recall that \eqn{\ref{eqn:oddP}} gives an explicit determinant-type formula
\[
\kP\od (\eps, \epst) = \frac{1}{d!}\sum_{|I|=|J|=d} \det A_{I,J}(\eps, \epst),
\]
where $A = \diag(\epst) V \diag(\eps)$. Recall $s_j(\eps) := \sum_{k=0}^{j-1} \eps_k$.
  
We also treat $\kP\od$ as 0 when the inputs $\eps, \epst\notin \ZeeII^{n+1}$.
\begin{prop} \label{prop:rec}
    %Let the notation be as above. 
    Let $\eps, \epst\in \ZeeII^{n+1}$ be such that $|\eps|=|\epst|$. Then the odd Krawtchouk polynomial $\kP\od(\eps, \epst)$ satisfies the following recurrence relations
    \begin{align} 
    \eps_i \kP\od (\eps, \epst) &= q_0^{-1}\qt_i \sum_{0\leq k,l\leq n}q_k v_{k,i} v_{l,i}(-1)^{s_k(\epst-v_l) + s_l(\epst)}\epst_l \kP\od (\eps, \epst+v_k-v_l), \label{eqn:ith}
\\
    \epst_i \kP\od (\eps, \epst) &= q_0^{-1}q_i \sum_{0\leq k,l\leq n}\qt_k v_{i,k} v_{i,l}(-1)^{s_k(\eps-v_l) + s_l(\eps)}\eps_l \kP\od (\eps+v_k-v_l, \epst).\label{eqn:ithT}
    \end{align}
We further have
    \begin{align} 
    \eps_0 \kP\od (\eps, \epst)  = \sum_{0\leq k\leq n} q_k \epst_k\kP\od(\eps, \epst) + \sum_{0\leq k\neq l \leq n} q_k (-1)^{s_k(\epst-v_l)+s_l(\epst)} \epst_l \kP \od(\eps, \epst+v_k -v_l), \label{eqn:eigen}\\
    \epst_0 \kP\od (\eps, \epst)  = \sum_{0\leq k\leq n} \qt_k \eps_k\kP\od(\eps, \epst) + \sum_{0\leq k\neq l \leq n} \qt_k (-1)^{s_k(\eps-v_l)+s_l(\eps)} \eps_l \kP \od(\eps +v_k -v_l, \epst). \label{eqn:eigent}
    \end{align} 
\end{prop}

\begin{proof}
    By Proposition~\ref{prop:sign}, for all $i$ we have
\[
\langle D_{i,i}\xi^\eps, \xit^{\epst}\rangle = \langle \xi^\eps, D_{i,i}\xit^{\epst}\rangle.
\]
Combined with \eqn{\ref{eqn:tCartan}} and Proposition~\ref{prop:tEijAct}, we compute the following
\begin{align*}
     \langle\xi^\eps,  D_{i,i} \xit^{\epst} \rangle &= q_0^{-1}\qt_i \sum_{0\leq k,l\leq n}q_k v_{k,i} v_{l,i}(-1)^{s_k(\epst-v_l) + s_l(\epst)} \epst_l \langle \xi^\eps, \xit^{\epst+v_k-v_l}\rangle \\
     & = \frac{d!}{q_0^d}  q_0^{-1}\qt_i \sum_{0\leq k,l\leq n}q_k v_{k,i} v_{l,i}(-1)^{s_k(\epst-v_l) + s_l(\epst)} \epst_l\kP\od (\eps, \epst+v_k-v_l).
\end{align*}
By Corollary~\ref{cor:pairing} and the fact that $\xi^\eps$ is an $\hk$-weight vector, we have $\langle D_{i,i}\xi^\eps, \xit^{\epst}\rangle =  \eps_i \frac{d!}{q_0^d} \kP\od (\eps, \epst)$. 
Thus \eqn{\ref{eqn:ith}} follows. By considering $\langle \tilde{D}_{i,i}\xi^\eps, \xit^{\epst}\rangle = \langle \xi^\eps, \tilde{D}_{i,i}\xit^{\epst}\rangle$, \eqn{\ref{eqn:ithT}} follows similarly.

To show \eqn{\ref{eqn:eigen}} and \eqn{\ref{eqn:eigent}}, we first use the condition $V\Qt V\Trn = q_0 Q^{-1}$ to obtain
\begin{equation} \label{eqn:sumLem}
    \sum_{i=0}^n \qt_i v_{k,i} v_{l,i} = \frac{q_0}{q_k} \delta_{kl}, \text{ and } \sum_{i=1}^n \qt_i v_{k,i} v_{l,i} = \frac{q_0}{q_k} \delta_{kl} - \qt_0.
\end{equation}
Let $d:=|\eps|=|\epst|$. Summing up \eqn{\ref{eqn:ith}} for $i = 1, \dots, n$, we get
\begin{align*}
    (d-\eps_0) \kP\od(\eps, \epst) &= q_0^{-1} \sum_{\substack{0\leq i,k,l \leq n\\ i\neq 0}} \qt_i q_k v_{k,i} v_{l,i} (-1)^{s_k(\epst-v_l)+s_l(\epst)} \epst_l \kP \od(\eps, \epst+v_k -v_l) \\
    & =q_0^{-1}\sum_{0\leq k,l\leq n} \left(\sum_{i = 1}^n \qt_i q_k v_{k,i} v_{l,i}\right ) (-1)^{s_k(\epst-v_l)+s_l(\epst)} \epst_l \kP \od(\eps, \epst+v_k -v_l) \\
    & = q_0^{-1} \left(\sum_{0\leq k \leq n} \left(q_0 - \qt_0 q_k\right )  \epst_k \kP \od(\eps, \epst) \right. \\
    + & \left. \sum_{0\leq k\neq l \leq n}  \left( - \qt_0 q_k\right )  (-1)^{s_k(\epst-v_l)+s_l(\epst)} \epst_l \kP \od(\eps, \epst+v_k -v_l) \right)\\
    & = d \kP\od (\eps, \epst) - \sum_{0\leq k \leq n} q_k \epst_k\kP\od(\eps, \epst) - \sum_{0\leq k\neq l \leq n} q_k (-1)^{s_k(\epst-v_l)+s_l(\epst)} \epst_l \kP \od(\eps, \epst+v_k -v_l).
\end{align*}
Here we used $q_0 = \qt_0$ and $s_k(\epst) =s_k(\epst-v_k)$. The above calculation gives \eqn{\ref{eqn:eigen}}. \eqn{\ref{eqn:eigent}} is similarly obtained by a tilde version of the argument. 
\end{proof}

\begin{rmk}
    Summing up \eqn{\ref{eqn:ith}} from 0 to $n$ and using \eqn{\ref{eqn:sumLem}}, we have the following tautological equation:
    \[
    d \kP\od(\eps, \epst) = q_0^{-1} \sum_{0\leq k, l \leq n} q_0\delta_{kl} (-1)^{s_k(\epst-v_l)+s_l(\epst)} \epst_l \kP \od(\eps, \epst+v_k -v_l) = d \kP\od(\eps, \epst).
    \]
    Indeed, the sum $\sum D_{i,i} \in \gl(m+1|n+1)$ acts as the total degree operator on $\xi$'s.
\end{rmk} 

We interpret these results as follows. 
Let $\tilde{A}_d$ (respectively $A_d$) denote the space of homogeneous degree $d$ polynomials linear in $\epst_j$ (respectively in $\eps_j$). 
Let $\cW = \Spn \{\xit_j:j\in [n]\} = \Spn \{\xi_j:j\in [n]\}$ and $\mathcal{F}^{(d)} := \bigwedge^d \cW$. Then $\mathcal{F}=\bigoplus_{j=0}^{n+1} \mathcal{F}^{(d)}$ is also known as the Fock space on $\cW$.
We define
\[
\tilde{\psi}: \tilde{A}_d \rightarrow \mathcal{F}^{(d)}: \sum_I c_I \epst_I \mapsto \sum_I c_I \xit_I, \quad {\psi}: A_d \rightarrow \mathcal{F}^{(d)}: \sum_I c_I \eps_I \mapsto \sum_I c_I \xi_I, 
\] 
where $I$ ranges over the strictly increasing index subsets of $[n]$ of size $d$, and $c_I \in \C$. Both are clearly linear isomorphisms.

Consider the action of the Cartan elements $D_{i,i}$ and $\tilde{D}_{i,i}$ on $\mathcal{F}^{(d)}$. Via $\rho$ the elements $D_{i,i}$ and $\tilde{D}_{i,i}$ act as the degree operators $\xi_i \partial_{\xi_i}$ and $\xit_i \partial_{\xit_i}$ respectively. 
By Lemma~\ref{lem:cartanSwap}, on $\mathcal{F}^{(d)}$ we have
\begin{equation} \label{eqn:diffOp}
    D_{i,i} = q_0^{-1} \qt_i \sum_{0\leq k, l \leq n} q_k v_{k,i} v_{l,i} \xit_k \partial_{\xit_l}, \quad \tilde{D}_{i,i} = q_0^{-1} q_i \sum_{0\leq k, l \leq n} \qt_k v_{i,k} v_{i,l} \xi_k \partial_{\xi_l}.
\end{equation}
Using \eqn{\ref{eqn:tEijAct}} we obtain the matrix entries of $\xit_k \partial_{\xit_l}$ with respect to the $\{\xit^\epst:|\epst|=d\}$ basis:
\[
\xit_k \partial_{\xit_l} (\xit^\epst) = \sum_{\etat} D_{\etat, \epst}^{k,l} \xit^\etat, \quad \text{where }  D_{\etat, \epst}^{k,l} = \epst_l (-1)^{s_k(\epst-v_l)+s_l(\epst)}\delta_{\etat, \epst+v_k -v_l},
\]
and it follows that $D_{\etat, \epst}^{k,l} = D_{\epst, \etat}^{l,k}$. 
Denote the transpose of $X$ with respect to the basis $\{\xit^\epst\}$ (respectively $\{\xi^\eps\}$) by $X^{\mkern-2mu\mathsf{t}, \xit}$ (respectively $X^{\mkern-2mu\mathsf{t}, \xi}$), and we have
\[
(\xit_k \partial_{\xit_l})^{\mkern-2mu\mathsf{t}, \xit} = \xit_l \partial_{\xit_k}, \quad \left(\xi_k \partial_{\xi_l}\right)^{\mkern-2mu\mathsf{t}, \xi} = \xi_l \partial_{\xi_k}.
\]
Thus by \eqn{\ref{eqn:diffOp}}, explicitly,
\begin{equation}
    D_{i,i}^{\mkern-2mu\mathsf{t}, \xit} = q_0^{-1} \qt_i \sum_{0\leq k, l \leq n} q_k v_{k,i} v_{l,i} \xit_l \partial_{\xit_k}, \quad \tilde{D}_{i,i}^{\mkern-2mu\mathsf{t}, \xi} = q_0^{-1} q_i \sum_{0\leq k, l \leq n} \qt_k v_{i,k} v_{i,l} \xi_l \partial_{\xi_k}.
\end{equation}
We also write
\[
\tilde{\kP}\od^\wedge (\eps) := \tilde{\psi}(\kP\od(\eps, -)), \quad {\kP}\od^\wedge (\epst) := {\psi}(\kP\od(-, \epst)).
\]
Under this identification, we may rewrite Proposition~\ref{prop:rec} as follows.
\begin{prop} \label{prop:wedgeRec}
%Let the notation be as above. We have
For $i\in [n]$, we have
\begin{equation}\label{eqn:wedgeiTh}
    D_{i,i}^{\mkern-2mu\mathsf{t}, \xit} \tilde{\kP}\od^\wedge (\eps) =  \eps_i\tilde{\kP}\od^\wedge (\eps),\quad  \tilde{D}_{i,i}^{\mkern-2mu\mathsf{t}, \xi} {\kP}\od^\wedge (\epst) =  \epst_i{\kP}\od^\wedge (\epst).
\end{equation}
Furthermore,
\begin{equation}\label{eqn:wedgeEigen}
    \left(d-\sum_{i=1}^n D_{i,i}^{\mkern-2mu\mathsf{t}, \xit}\right) \tilde{\kP}\od^\wedge (\eps)  = \eps_0\tilde{\kP}\od^\wedge (\eps) ,\quad \left(d-\sum_{i=1}^n \tilde{D}_{i,i}^{\mkern-2mu\mathsf{t}, \xi}\right) {\kP}\od^\wedge (\epst)  = \epst_0 {\kP}\od^\wedge (\epst).
\end{equation}
    
\end{prop}
\begin{proof}
For the first identity in \eqn{\ref{eqn:wedgeiTh}}, write
\[
D_{i,i}\xit^\etat = \sum_{\epst} D_{\epst,\etat} \xit^{\epst}
\]
where $D_{\epst,\etat} = q_0^{-1} \qt_i \sum_{0\leq k, l\leq n} q_k v_{k,i} v_{l,i} (-1)^{s_k(\etat-v_l) + s_l(\etat)}\etat_l \delta_{\epst, \etat+v_k-v_l}$ by Proposition~\ref{prop:tEijAct}.
Suppose
\[
\tilde{\kP}\od^\wedge (\eps) = \sum_{\etat}c_\etat \xit^\etat.
\]
Then %we calculate that
\begin{align*}
    D_{i,i}^{\mkern-2mu\mathsf{t}, \xit} \tilde{\kP}\od^\wedge (\eps) &= \sum_{\etat, \epst} c_\etat D_{\etat, \epst} \xit^\epst \\
    &= q_0^{-1} \qt_i \sum_{\epst}  \sum_{0\leq k, l \leq n}   q_k v_{k,i} v_{l,i} (-1)^{s_k(\epst-v_l) + s_l(\epst)}\epst_l c_{\epst+v_k-v_l}\xit^\epst,
\end{align*}
which is exactly $\eps_i \tilde{\kP}\od^\wedge (\eps)$ by \eqn{\ref{eqn:ith}} with $\tilde{\psi}$ applied on both sides. 
The second identity is analogous. 
Finally, as above, summing \eqn{\ref{eqn:wedgeiTh}} for $i = 1, \dots, n$ gives \eqn{\ref{eqn:wedgeEigen}}.
\end{proof}

From the point of view of exterior algebra, the odd Krawtchouk polynomials as elements $\tilde{\kP}\od^\wedge (\eps)$ and ${\kP}\od^\wedge (\epst)$ in $\mathcal{F}^{(d)}$, become eigenvectors of naturally commuting operators $D_{i,i}^{\mkern-2mu\mathsf{t}, \xit}$ and $\tilde{D}_{i,i}^{\mkern-2mu\mathsf{t}, \xi}$ which arise from the actions of Cartan subalgebras. 
Also, the dimension of $\mathcal{F}^{(d)}$ is $\binom{n+1}{d}$, which is exactly the cardinality of $\{\kP\od(\eps,-): |\eps|=d\}$.
By orthogonality, they must be linearly independent and hence form a basis for $\mathcal{F}^{(d)}$. By applying $\tilde{\psi}^{-1}$ and $\psi^{-1}$ respectively, we see that for fixed $\eps$ or $\epst$, the odd Krawtchouk polynomials $\kP\od(\eps, -)$ and $\kP\od(-,\epst)$ form bases for $\tilde{A}_d$ and $A_d$ respectively. 

\begin{rmk}
    In our work, we use a finite Fock space framework similar to that used by Borasi in \cite{B22}, up to a harmless shift in indexing. Our operators $\xi_k\partial_{\xi_l}$ (similarly $\xit_k \partial_{\xit_l}$) can be written as $c_k^\dagger c_l$ where $c_k^\dagger$ is the creation operator given by exterior multiplication by $\xi_k$, and $c_l$ is the annihilation operator given by contraction. Borasi uses the spin representation of $\mathfrak{so}(2n+1,\mathbb C)$ to study finite-dimensional systems of free fermions and then lifts the corresponding fermionic Hamiltonian to a differential operator on the spin group, eventually relating the Euclidean-time evolution to stochastic diffusion processes on the spin group.
\end{rmk}

\section{Spherical functions} \label{sec:zonal}
In this section, we give an interpretation of the odd Krawtchouk polynomials of a fixed degree $d\in \{0, 1, \dots, n+1\}$ using the language of fermionic Fock space and (zonal) spherical functions. We assume in our initial data $\Lambda=(q, \qt, V)$ that $q$ and $\qt$ are positive real, and $V$ is real.

We first give a motivation from the physics point of view.    
Let $\cWR = \R^{n+1}$ and denote the standard basis for $\cWR$ by $\{\xi_j\}_{j = 0}^{n}$. 
The associated fermionic Fock space is
\[
\mathcal{F}=\bigoplus_{j=0}^{n+1}\bigwedge^j\cWR.
\]
If we regard $\cWR$ as the one-particle state space for a fermion with $(n+1)$ available modes, then the $d$-th component $\mathcal{F}^{(d)} = \bigwedge^d \cWR$ may be interpreted as the state space of $d$ identical fermions occupying $(n+1)$ modes.
The antisymmetry of the exterior products enforces the Pauli exclusion principle ($\xi_i\wedge \xi_i=0$) satisfied by fermions. 
As usual, we omit wedges when the anticommutativity is clear.

Given any oriented orthonormal basis $\{f_0, f_1, \dots,f_{d-1}\}$ for a subspace of $\cWR$, the wedge product $f_0f_1\cdots f_{d-1} = f_0\wedge f_1 \wedge \cdots \wedge f_{d-1}$ is sometimes called a Slater determinant (see \cite{AdS20}) in quantum physics.
Let $G := SO(n+1)$ and $K := SO(d)\times SO(n+1-d)$.
Consider the oriented Grassmannian 
\[
\oGr(n+1, d) := SO(n+1)/ (SO(d)\times SO(n+1-d)) = G/K.
\]
Such a wedge $f_0f_1\cdots f_{d-1}$ determines an oriented $d$-dimensional subspace of $\cWR$ spanned by $f_i$'s. 
Thus, any non-zero simple wedge up to a positive scalar naturally corresponds to a point in $\oGr(n+1,d)$. 

Moreover, given an element $g\in G$, its $(I,J)$-th minors appear as the Pl\"{u}cker coordinates in the following sense. View $\cWR$ as the natural $(n+1)$-dimensional representation of $G$. The action naturally induces a $G$-module structure on $\mathcal{F}$ via
\[
g.(u_0 \cdots  u_{d-1}) := (g.u_0)  \cdots  (g.u_{d-1}).
\]
For an ordered index subset $I = \{i_k:i_0<i_1 < \cdots < i_{d-1}\}$ of $[n]$, set $\xi_I := \xi_{i_0} \xi_{i_1}  \cdots  \xi_{i_{d-1}}$. Then
\begin{equation} \label{eqn:gij}
    g . \xi_J = \sum_{|I|=d} \det(g_{I,J}) \xi_I.
\end{equation}
Equivalently, for each fixed $J$, these coefficients are the Pl\"{u}cker coordinates of the oriented $d$-dimensional subspace $gP_J$, where $P_J=\Spn\{\xi_j:j\in J\}$.

For each fixed index subset $J$ with $|J|=d$, and $g\in SO(n+1)$ the vector $g.\xi_J$ in \eqn{\ref{eqn:gij}} has norm $1$ since $g\in SO(n+1)$. Hence the quantities
\begin{equation} \label{eqn:prob}
    \mathbb{P}_{I\mid J}:=|\det(g_{I,J})|^2
\end{equation}
form a probability distribution on the set of size-$d$ index subsets $\{I\subseteq [n]: |I|=d\}$. In the fermionic interpretation of $\wedge^d\cWR$, this is the probability distribution obtained by expressing the state $g.\xi_J$ in the occupation basis $\{\xi_I: |I|=d\}$, cf. \cite[Section~2]{Gottlieb07}.

We now introduce some representation-theoretic machinery. 
The pair $(G, K)$ is a Gelfand pair.
For the following result, see \cite[Section~19.2]{FH}, where explicit highest weights are given.

\begin{prop} \label{prop:FH}
    The $G$-module $\wedge^d \cWR$ is always irreducible, except when $n+1$ is even and $d = \dfrac{n+1}{2}$, in which case $\wedge^d \cWR$ is the sum of two irreducible components. 
\end{prop}

Consider the $K$-action on $\wedge^d \cWR$. Observe that in $\wedge^d \cWR$, the vector
\[
v^K = \xi_0 \xi_1  \cdots  \xi_{d-1}
\]
is $K$-fixed. Indeed, for $(k_0, k_1)\in SO(d)\times SO(n+1-d)$, we have
\[
(k_0, k_1).v^K = k_0 ( \xi_0 \xi_1  \cdots  \xi_{d-1}) = (\det k_0 )v^K= v^K.
\]
When $\wedge^d \cWR$ is irreducible, this spherical vector is unique (cf. \cite[Section~4]{Zhu25}). When $d = (n+1)/2 \in \Z_{>0}$, there is another $K$-fixed vector
\[
v_-^K =  \xi_d \cdots  \xi_n.
\]
In this case $\wedge^d \cWR$ splits exactly into two components (Proposition~\ref{prop:FH}), each containing a one-dimensional $K$-fixed subspace. We choose the component containing $\xi_0  \cdots  \xi_{d-1}$.

Consider the spherical functions $\phi$ associated with $\oGr$. They arise precisely as the matrix coefficients of irreducible representations of $G$. If an irreducible module $M$ of $G$ is also spherical with a normalized spherical vector $u^K$, then $\phi(g) = \langle u^K, g.u^K \rangle$ for any $g\in G$, cf. \cite{Helgason2}.
Explicitly, in our context, for each $d$ we may record a spherical function
\[
\phi_d(g) = ( v^K, g.v^K )
\]
where the pairing $( \cdot, \cdot )$ is extended naturally from the standard inner product on $\cWR$. %Here we take the convention  that we do not distinguish a spherical function on $G$ or on $G/K$. 

We now return to the original data $\Lambda = (q, \qt, V)$ defining $\kP\od$. In our context, $\Lambda$ encodes a basis change in the modes of the Fock space $\mathcal{F}$. Let
\begin{equation} \label{eqn:gV}
    g = g_V := q_0^{-1/2}Q^{1/2}V\Qt^{1/2}.
\end{equation}
Since $QV\Qt V\Trn = q_0 I_{n+1}$, we see that $g \in O(n+1)$. Moreover, we may choose $Q^{1/2}$ and $\Qt^{1/2}$ so that $\det g = 1$. Thus any given $V$ determines an element in $SO(n+1)$. Although it is not unique, our formulation below is independent of this choice.

Let $J$ be any index subset of size $d$. Since $G$ acts transitively on orthonormal bases for $\cWR$, we can pick $\sigma_I, \sigma_J \in G$ such that 
\[
\sigma_I.v^K =  \xi_I, \quad \sigma_J.v^K = \xi_J.
\]
It is not hard to see that the $(I,J)$-minor of $g$ is exactly given by $( \xi_I, g.\xi_J )$. Therefore, the $(I, J)$-minor of $g$ can be rewritten in terms of $\phi_d$:
\begin{align}
    (  \xi_I, g.\xi_J ) & = (  \sigma_I v^K, g\sigma_J.v^K ) \notag \\
    &= (  v^K, \sigma_I^{-1}g\sigma_J.v^K ) \notag \\
    &= \phi_d(\sigma_I^{-1}g\sigma_J). \label{eqn:gIJ}
\end{align}
If $\sigma_I'$ is another choice such that $\sigma_I' .v^K = \xi_I$, then we must have
\[
\sigma_I^{-1}\sigma_I'.v^K = v^K,
\]
meaning $\sigma_I^{-1}\sigma_I' \in K$. Similarly for $\sigma_J$. By the $K$-bi-invariance of $\phi$, 
\[
\phi_d(\sigma_I^{-1}g\sigma_J) = \phi_d(\sigma_I^{-1}\sigma_I'(\sigma_I')^{-1}g\sigma_J).
\]
%thus the choices of $\sigma_I$, and similarly of $\sigma_J$, do not matter. 
Thus $\phi_d(\sigma_I^{-1}g\sigma_J)$ is independent of the choices of $\sigma_I$ and $\sigma_J$.
Additionally, when $\bigwedge^d\cWR$ is not irreducible, this formula still works, and we may choose $v^K_-$ to write down a different one.

Combining \eqn{\ref{eqn:oddP}}, \eqn{\ref{eqn:gV}}, and \eqn{\ref{eqn:gIJ}}, we have
\begin{equation} \label{eqn:krZonal}
    \kP\od (\eps, \epst) = \frac{q_0^{d/2}}{d!}\sum_{\substack{|I|=|J|=d}} q_I^{-\frac12} \qt_J^{-\frac12}\phi_d(\sigma_I^{-1}g_V\sigma_J) \epst_I\eps_J.
\end{equation}
This is the main conclusion of this section. \eqn{\ref{eqn:krZonal}} shows that the odd Krawtchouk polynomial $\kP\od(\eps, \epst)$ has coefficients given by weighted evaluation of $\phi_d$, the spherical function, associated with $\oGr(n+1, d)$, which appears naturally when we consider the fermionic Fock space. Additionally, in view of the probability interpretation above (\eqn{\ref{eqn:prob}}), the squared norms of the evaluation of $\phi_d$ are the transition probabilities of different occupation basis dictated by $\Lambda=(q, \qt, V)$.

\bibliographystyle{amsalpha}
\bibliography{ref}

\end{document}